\documentclass[leqno,11pt,a4paper]{article}
\usepackage[left=2.0cm, top=2.2cm,bottom=2.5cm,right=2.0cm]{geometry}
\usepackage{amsmath,amssymb,amsthm,mathrsfs,calc,graphicx}
\usepackage{tikz}
\usepackage[british]{babel}

\newtheorem {theorem}{Theorem}[section]
\newtheorem {proposition}[theorem]{Proposition}
\newtheorem {lemma}[theorem]{Lemma}
\newtheorem {corollary}[theorem]{Corollary}

{\theoremstyle{definition}

}
{\theoremstyle{theorem}
\newtheorem {remark}[theorem]{Remark}

}

\newcommand{\Cov}{\operatorname{cov}}

\newcommand{\Var}{\operatorname{var}}

\def\ba{\begin{array}}
\def\ea{\end{array}}
\def\bea{\begin{eqnarray} \label}
\def\eea{\end{eqnarray}}
\def\be{\begin{equation} \label}
\def\ee{\end{equation}}
\def\bit{\begin{itemize}}
\def\eit{\end{itemize}}
\def\ben{\begin{enumerate}}
\def\een{\end{enumerate}}

\def\DD{\mathbb{D}}
\def\EE{\mathbb{E}}

\def\NN{\mathbb{N}}
\def\PP{\mathbb{P}}

\def\RR{\mathbb{R}}

\def\ZZ{\mathbb{Z}}

\def\a{\alpha}

\def\g{\gamma}

\def\e{\varepsilon}

\def\l{\lambda}
\def\r{\rho}
\def\s{\sigma}

\def\O{\Omega}

\def\bF{\mathbf{F}}

\def\bX{\mathbf{X}}
\def\bY{\mathbf{Y}}
\def\bZ{\mathbf{Z}}

\def\cA{\mathcal{A}}
\def\cB{\mathcal{B}}

\def\cF{\mathcal{F}}

\def\cH{\mathcal{H}}

\def\cN{\mathcal{N}}

\def\H{\mathfrak{H}}

\def\dint{\textup{d}}

\def\var{{\textup{var}}}
\def\dom{{\textup{dom}}}
\def\s{\otimes}
\def\ts{\widetilde{\otimes}}

\def\fH{\mathfrak{H}}
\def\fC{\mathfrak{C}}
\def\cov{{\textup{cov}}}

\def\one{\mathbf 1}

\begin{document}

\title{\bfseries A new quantitative central limit theorem on the \\ Wiener space with applications to Gaussian processes}

\author{Tobias Fissler\footnotemark[1] \and Christoph Th\"ale\footnotemark[2]}

\date{}
\renewcommand{\thefootnote}{\fnsymbol{footnote}}

\footnotetext[1]{University of Bern, Department of Mathematics and Statistics, Institute of Mathematical Statistics and Actuarial Science, CH-3012 Bern, Switzerland. E-mail: \texttt{tobias.fissler@stat.unibe.ch}}
\footnotetext[2]{Ruhr University Bochum, Faculty of Mathematics, D-44780 Bochum, Germany. E-mail: \texttt{christoph.thaele@rub.de}}

\maketitle

\begin{abstract}
Quantitative limit theorems for non-linear functionals on the Wiener space are considered. Given the possibly infinite sequence of kernels of the chaos decomposition of such a functional, an estimate for different probability distances between the functional and a Gaussian random variable in terms of contraction norms of these kernels is derived. The applicability of this result is demonstrated by means of the Breuer-Major theorem, unfolding thereby a new connection between the Hermite rank of the considered function and a chaotic gap. Especially, power variations of the fractional Brownian motion and processes belonging to the Cauchy class are studied.
\end{abstract}

\section{Introduction}

Central limit theorems for non-linear functionals of Gaussian random processes (or measures) have triggered an enormous development in probability theory and mathematical statistics during the last decade. A cornerstone in this new area is the so-called fourth moment theorem of Nualart and Peccati. It says that a sequence $I_q(f^{(n)})$ of Gaussian multiple stochastic integrals of a fixed order $q\geq 2$ satisfying the normalization condition $\EE[I_q(f^{(n)})^2]=1$ for all $n\geq 1$ converges in distribution, as $n\to\infty$, to a standard Gaussian random variable $Z$ if and only if the sequence  $\EE[I_q(f^{(n)})^4]$ of their fourth moments converges to $3$, the fourth moment of $Z$. This qualitative limit theorem has been extended by Nourdin and Peccati in \cite{NourdinPeccati09} to a quantitative statement in that the distance between the laws of $I_q(f^{(n)})$ and $Z$ is measured in a suitable probability metric. For example, the total variation distance $d_{TV}(I_q(f^{(n)}),Z)$ between $I_q(f^{(n)})$ and the Gaussian variable $Z$ can be bounded from above by
\begin{equation}\label{eq:RateIntro}
d_{TV}(I_q(f^{(n)}),Z) \leq C\sqrt{\EE[I_q(f^{(n)})]-3}
\end{equation}
with a constant $C\in(0,\infty)$ only depending on $q$. More recently, Nourdin and Peccati \cite{NourdinPeccatiOptimalRates} derived the optimal rate of convergence, removing thereby the square-root in \eqref{eq:RateIntro}. We emphasize that the proof of the estimate \eqref{eq:RateIntro} is based on a combination of Stein's method for normal approximation with the Malliavin calculus of variations on the Wiener space. For further information and background material, we refer the reader to the monograph \cite{NourdinPeccatiBook}.

\medbreak

While the Malliavin-Stein approach provides useful estimates in case of a sequence of random elements living inside a fixed Wiener chaos or inside a finite sum of Wiener chaoses, the bounds become less tractable in cases in which the functionals belong to an infinite sum of Wiener chaoses, that is, if the functional $F$ has the representation
\begin{equation}\label{eq:ChaoticDecompsitionIntro}
F = \sum_{q=0}^\infty I_q(f_q)
\end{equation}
with infinitely many of the functions $f_q$ (called kernels in the sequel) being non-zero.  On the other hand, functionals of this type often appear in concrete applications. Distinguished examples are the number of zeros of a random trigonometric polynomial \cite{AzaisRandTrigPol}, the power and the bi-power variation of a Gaussian random process \cite{BarndorffNielsenEtAl1}, the Euler characteristic of a Gaussian excursion set \cite{EstradeLeon} or the statistics appearing around the Breuer-Major theorem \cite{BiermeBonamiLeon2011,NourdinPeccatiPodolskij}, to name just a few. One way to obtain quantitative central limit theorems in these cases is to apply the so-called second-order Poincar\'e inequality developed by Nourdin, Peccati and Reinert \cite{NourdinPeccatiReinertSecondOrderPoincare}. This method has the advantage that it is not necessary to specify the chaos decomposition of $F$ as at \eqref{eq:ChaoticDecompsitionIntro} explicitly, that is, to compute the functions $f_q$ there explicitly. On the other hand, a major drawback of this approach is that it typically leads to a suboptimal rate of convergence. Moreover, in many situations the kernels $f_q$ are in fact explicitly known and for this reason it is natural to ask for a purely analytical upper bound on the probability distance between $F$ and $Z$ in terms of the sequence of kernels $f_q$. The main goal of the present paper is to provide such an estimate (also for probability metrics different from the total variation distance) and to demonstrate its applicability by means of representative examples related to the classical Breuer-Major theorem. More precisely, we shall look at random variables of the form
$$
F_n={1\over\sqrt{n}}\sum_{k=1}^n\{g(X_k)-\EE[g(X_k)]\}\,,
$$
where $X=(X_k)_{k\in\ZZ}$ is a stationary Gaussian process and $g:\RR\to\RR$ is a suitable function. For example, $X$ could be obtained from the increments of a fractional Brownian motion and $g(x)=|x|^p-\EE|X_1|^p$, $p>0$, in which case $F_n$ becomes a so-called power variation of the fractional Brownian motion. In this context, our quantitative central limit theorem for $F_n$ unfolds a new and unexpected feature, namely that the rates of convergence are influenced by the interplay of the Hermite rank of the function $g$ and what we call the chaotic gap of $F_n$ (in addition to the asymptotic behavior of the covariance function of $X$, of course). We would like to emphasize that in the context of power variations of a fractional Brownian motion we will show that the rate of convergence in the central limit theorem is universal, that is, independent of the parameter $p$, and coincides with the known rate for the quadratic variation, where $p=2$. The same phenomenon also applies to processes that belong to the Cauchy class.

\medskip

Our text is structured as follows. In Section \ref{sec:GaussianAnalysis}, we summarize some basic elements of Gaussian analysis and, in particular, recall the definitions of the four basic operators from Malliavin calculus that are crucial for our theory. Our main result, Theorem \ref{thm:Main}, is presented in Section \ref{sec:MainResults}. Our applications to Gaussian random processes are the content of Section \ref{sec:Application}. Finally, in Section \ref{sec:Multivariate} we present a multivariate exteions of our main result. The Appendix gathers some technical lemmas.

\section{Elements of Gaussian analysis}\label{sec:GaussianAnalysis}

\subsection{Wiener chaos, chaos decomposition and multiplication formula}

We let $\fH$ be a real separable Hilbert space with inner product $\langle\,\cdot\,,\,\cdot\,\rangle_\fH$ and norm $\|\,\cdot\,\|_\fH$. Moreover, for integers $q\geq 1$ we denote by $\fH^{\otimes q}$ the $q$th tensor power and by $\fH^{\odot q}$ the $q$th symmetric tensor power of $\fH$. The space $\fH^{\otimes q}$ is supplied with the canonical scalar product $\langle\,\cdot\,,\,\cdot\,\rangle_{\fH^{\otimes q}}$ and the canonical norm $\|\,\cdot\,\|_{\fH^{\otimes q}}$, while $\fH^{\odot q}$ is equipped with the norm $\sqrt{q!}\|\,\cdot\,\|_{\fH^{\otimes q}}$.

An isonormal Gaussian process $W=\{W(h):h\in\fH\}$ over $\fH$ is a family of Gaussian random variables defined on a common probability space $(\O,\cF,\PP)$ and indexed by the elements of $\fH$ such that
$$
\EE[W(h)]=0\qquad\text{and}\qquad \EE[W(h)W(h')] = \langle h,h'\rangle_{\fH}\,,\qquad h,h'\in\fH\,.
$$
In what follows we will implicitly assume that the $\sigma$-field $\cF$ is generated by $W$, that is, $\cF=\sigma(W)$. Let us write $L^2(\O)$ for the space of square-integrable functions over $\O$.
For integers $q\geq 1$ we denote by $\fC_q$ the $q$th Wiener chaos over $\fH$. That is, $\fC_q$ is the closed linear subspace of $L^2(\O)$ generated by random variables of the form $H_q(W(h))$. Here, $H_q$ is the $q$th Hermite polynomial and $h\in\fH$ satisfies $\|h\|_\fH=1$. Recall that $H_0\equiv 0$ and that
\begin{align}\label{eq:Hermite def}
H_q(x) &= (-1)^q \exp(x^2/2)\,\frac{\dint^q}{\dint x^q}\exp(-x^2/2), \qquad q\ge1\,,\\
\EE[H_q(X)H_p(Y)] &=\begin{cases}
p!\,(\EE[XY])^p &: p=q\\
0 &: \text{otherwise}\,,
\end{cases}
\label{eq:Hermite moments}
\end{align}
for jointly Gaussian $X,Y$ and integers $p,q\geq 1$.
For convenience, we also define $\fC_0:=\RR$. The mapping $h^{\otimes q}\mapsto H_q(W(h))$ can be extended to a linear isometry, denoted by $I_q$, from $\fH^{\odot q}$ to the $q$th Wiener chaos $\fC_q$, see Chapter 2 in \cite{NourdinPeccatiBook}. We put $I_q(h):= I_q(\tilde h)$ for general $h\in\fH^{\otimes q}$ where $\tilde h\in\fH^{\odot q}$ is the canonical symmetrization of $h$, and we use the convention that $I_0\colon \RR\to\RR$ is the identity map. In particular, if $\fH=L^2(A)$ with a $\sigma$-finite non-atomic measure space $(A,\cA,\mu)$, then $I_q$ possesses an interpretation as a multiple stochastic integral of order $q$ with respect to the Gaussian random measure on $A$ with control measure $\mu$. We refer to Chapter 2.7 in \cite{NourdinPeccatiBook} for further details and explanations.

According to Theorem 2.2.4 in \cite{NourdinPeccatiBook}, every $F\in L^2(\O)$ admits a chaotic decomposition. In par\-ti\-cu\-lar, this means that
$$
F=\sum_{q=0}^\infty I_q(h_q)
$$
with $h_0=\EE[F]:=\int F\,\dint\PP$ and uniquely determined elements $h_q\in\fH^{\odot q}$, $q\geq 1$, that are called the kernels of the chaotic decomposition. We also mention that, for $q\geq 1$, $\EE[I_q(h_q)]=0$ and that
\begin{equation}\label{eq:Isometry}
\EE[I_p(h_p)I_q(h_q)]=\begin{cases}
p!\langle h_p,h_q\rangle_{\fH^{\otimes q}} &: p=q\\
0 &: \text{otherwise}
\end{cases}
\end{equation}
for $h_p\in\fH^{\odot p}$, $h_q\in\fH^{\odot q}$, $p,q\geq 1$, which implies that the variance of $F$ satisfies
\begin{equation}\label{eq:Variance}
\var(F):=\EE[F^2]-(\EE[F])^2=\sum_{q=1}^\infty q!\|h_q\|_{\fH^{\otimes q}}^2\,.
\end{equation}
More generally, the covariance of $F=\sum_{q=0}^\infty I_q(h_q)\in L^2(\O)$ and $G=\sum_{q=0}^\infty I_q(h_q')\in L^2(\O)$ is given by
\begin{equation}\label{eq:Covariance}
\cov(F,G) := \EE[FG]-\EE[F]\EE[G] = \sum_{q=1}^\infty q!\langle h_q,h_q'\rangle_{\fH^{\otimes q}}\,.
\end{equation}

Another crucial fact is that the product of two multiple stochastic integrals can be expressed as a linear combination of multiple stochastic integrals. More generally, let $p,q\geq 1$ be integers and $h\in\fH^{\odot p}$, $h'\in\fH^{\odot q}$. Then one has the multiplication formula
\begin{equation}\label{eq:MultiplicationFormula}
I_q(h)I_p(h') = \sum_{r=0}^{\min(p,q)}r!{q\choose r}{p\choose r}I_{p+q-2r}(h\ts_r h')\,,
\end{equation}
where $h\ts_r h':=\widetilde{h\s_r h'}$ stands for the canonical symmetrization of the contraction $h\s_r h'\in\fH^{\otimes p+q-2r}$. Note that for $h = h_1 \otimes \cdots \otimes h_p\in\fH^{\otimes p}$ and $h'= h'_1\otimes \cdots \otimes h'_q\in\fH^{\otimes q}$ the contraction can be defined as
\be{eq:contraction}
h\s_r h':= \langle h_1, h'_1\rangle_{\fH} \cdots \langle h_r, h'_r\rangle_{\fH}\,[h_{r+1}\otimes \cdots \otimes h_{p}\otimes h'_{r+1} \otimes\cdots\otimes  h'_{q}]\,.
\ee
By linearity, the contraction operation can be extended to any $h\in\fH^{\otimes p}$ and $h'\in\fH^{\otimes q}$.
In the case that $\fH=L^2(A)$ with a $\sigma$-finite non-atomic measure space $(A,\cA,\mu)$, we have that $\fH^{\otimes q}= L^2(A^q) := L^2(A^q, \cA^{\otimes q}, \mu^{\otimes q})$ and that  $\fH^{\odot q}$ coincides with the space $L^2_{\rm sym}(A^q)$ of $\mu^{\otimes q}$-almost everywhere symmetric functions on $A^q$. Moreover,
\begin{align*}
(f\s_r g) (y_1,\ldots,y_{p+q-2r}) := \int_{A^r} & f(x_1,\ldots,x_r,y_1,\ldots,y_{p-r})\\
& \times g(x_1,\ldots,x_r,y_{p-r+1},\ldots,y_{p+q-2r})\,\mu^{\otimes r}(\dint(x_1,\ldots,x_r))
\end{align*}
with $f\in L^2_{\rm sym}(A^p)$, $g\in L^2_{\rm sym}(A^q)$ and $y_1,\ldots,y_{p+q-2r}\in A$.

\paragraph{Convention.} Through our paper, we will adopt the following convention that whenever the Hilbert space $\fH$ coincides with an $L^2(A)$-space we write $f$ instead of $h$ for an element of $L^2(A)$ to underline that we are dealing with functions. Furthermore, we use the shorthand notation $\|\,\cdot\,\|_q$ for $\|\,\cdot\,\|_{\fH^{\otimes q}}$ for all integers $q\geq 1$.

\subsection{Malliavin operators}

In this section, we recall the definition of the four basic operators from Malliavin calculus and summarize those properties which are needed later. For that purpose and to simplify our presentation we assume from now on that $\fH=L^2(A)$ with a $\sigma$-finite non-atomic measure space $(A,\cA,\mu)$. Note that because of isomorphy of Hilbert spaces, this is no restriction of generality. For further details we direct the reader to the monographs \cite{Janson,NourdinPeccatiBook,Nualart}.

\paragraph{Malliavin derivative.} Suppose that $F\in L^2(\O)$ has a chaos decomposition
\begin{equation}\label{eqn:chaotic}
F=\sum_{q=0}^\infty I_q(f_q), \qquad f_q\in L^2_{\rm sym}(A^q)\,,
\end{equation}
and suppose that $\sum\limits_{q=1}^\infty q\,q!\|f_q\|_q^2<\infty$. In this case we say that $F$ belongs to the domain of $D$, formally we indicate this by writing $F\in{\rm dom}(D)$. For $F\in{\rm dom}(D)$ and $x\in A$ we define the Malliavin derivative of $F$ in direction $x$ as
\begin{equation}\label{eq:DefD}
D_xF := \sum_{q=1}^\infty qI_{q-1}(f_q(x,\,\cdot\,))\,,
\end{equation}
where $f_q(x,\,\cdot\,)\in L_{\rm sym}^2(A^{q-1})$ stands for the function $f_q$ with one of its variables fixed to be $x$ (which one is irrelevant, since the functions $f_q$ are symmetric). 

We further define for all integers $k\geq 1$ the iterated Malliavin derivative $D^kF$ as 
$$
D_{x_1,\ldots,x_k}^kF:=\sum_{q=k}^\infty q(q-1)\cdots(q-k+1)\,I_{q-k}(f_q(x_1,\ldots,x_k,\,\cdot\,))\,,\qquad x_1,\ldots,x_k\in A\,,
$$
whenever $F\in{\rm dom}(D^k)$, that is, if $F=\sum_{q=0}^\infty I_q(f_q)$ satisfies $\sum_{q=k}^\infty q(q-1)\cdots(q-k+1)\|f_q\|_q^2<\infty$.

Finally, we introduce the subspace $\DD^{1,4}$ of ${\rm dom}(D)$ containing all $F\in L^4(\O)$ such that 
\[
\EE\|DF\|_1^4 = \EE\,\Big| \int_A |D_xF|^2\, \mu(\dint x)\Big|^2<\infty\,,
\]
see Chapter 2.3 in \cite{NourdinPeccatiBook} for a formal construction. Moreover, we recall that the Malliavin derivative can be used to compute the kernels $f_q$ in the chaotic decomposition of a given functional $F$. Namely, assuming that $F\in{\rm dom}(D^q)$ for some $q\geq 1$, Stroock's formula \cite[Corollary 2.7.8]{NourdinPeccatiBook} says that
\be{eq:Stroock}
f_q = {1\over q!}\, \EE[D^qF]\,.
\ee

\paragraph{Divergence.} We write $L^2(A\times\O):= L^2(A\times \O, \cA \otimes \mathcal F, \mu \otimes \PP)$ for the space of square-integrable random processes $u=(u_x)_{x\in A}$ on $A$. Fix such a process $u\in L^2(A\times\O)$ and suppose that it satisfies
$$
\Big|\EE\int_A (D_xF)\,u_x\,\mu(\dint x)\Big|\leq c\,\EE[F^2]
$$
for all $F\in{\rm dom}(D)$ and some constant $c>0$ that is allowed to depend on $u$. We denote the class of such processes by ${\rm dom}(\delta)$ and define for $u\in{\rm dom}(\delta)$ the divergence $\delta(u)$ of $u$ by the duality relation
$$
\EE[F\delta(u)] = \EE\int_A (D_xF)\,u_x\,\mu(\dint x)\,,\qquad F\in{\rm dom}(D)\,.
$$
That is, $\delta$ is the operator which is adjoint to the Malliavin derivative $D$.

The divergence can also be defined in terms of chaotic decompositions. Suppose that $u\in \dom(\delta)$ such that $u_x\in L^2(A\times \O)$ for all $x\in A$. Then there are kernels $f_q\in L^2(A^{q+1})$, $q\ge0$, such that 
\[
u_x = \sum_{q=0}^\infty I_q(f_q(x,\,\cdot\,))\,, \qquad x\in A\,,
\]
and $f_q(x,\,\cdot\,)\in L_{\rm sym}^2(A^q)$. Moreover, $u\in\dom(\delta)$ if and only if $\sum\limits_{q=0}^\infty (q+1)!\|\tilde f_q\|^2<\infty$ and in this case $\delta(u)$ is given by
\[
\delta(u) = \sum_{q=0}^\infty I_{q+1} (\tilde f_q)\,,
\]
where 
\[\tilde f_q(x_1, \ldots, x_{q+1}) := \frac{1}{(q+1)!}\sum_{\pi}f(x_{\pi(1)},\ldots,x_{\pi(q+1)})
\]
denotes the canonical symmetrization of $f_q\in L^2(A^{q+1})$ with the sum running over all permutations $\pi$ of $\{1,\ldots,q+1\}$.

\paragraph{Ornstein-Uhlenbeck generator and its pseudo-inverse.} Let $F\in L^2(\O)$ be a square integrable functional with chaos decomposition as at \eqref{eqn:chaotic} and define
$$
LF := -\sum_{q=0}^\infty qI_q(f_q)\,,
$$
whenever the series converges in $L^2(\O)$. The domain ${\rm dom}(L)$ of $L$ is the set of those $F\in L^2(\O)$ for which $\sum\limits_{q=1}^\infty q^2\,q!\|f_q\|_q^2<\infty$. The operator $L$ is called the generator of the Ornstein-Uhlenbeck semigroup associated with the Gaussian random measure on $A$ having control measure $\mu$. By $L^{-1}$ we denote its pseudo-inverse acting on centred $F\in L^2(\O)$ as follows:
\begin{equation}\label{eq:defL-1}
L^{-1}F := -\sum_{q=1}^\infty {1\over q}I_q(f_q)\,.
\end{equation}
For non-centred $F\in L^2(\O)$ we put $L^{-1}F:=L^{-1}(F-\EE[F])$. Clearly, for centred $F\in{\rm dom}(L)$ one has that $LL^{-1}F=L^{-1}LF=F$. Moreover, the operators $D$, $\delta$ and $L$ are related by
$$
\delta(DF) = -LF\,,\qquad F\in {\rm dom}(L)\,.
$$
In fact, according to \cite[Proposition 1.4.8]{Nualart}, $F\in{\rm dom}(L)$ is equivalent to $F\in{\rm dom}(D)$ and $DF\in{\rm dom}(\delta)$.

\section{A quantitative central limit theorem}\label{sec:MainResults}

Let $W$ be an isonormal Gaussian process defined on a probability space $(\O,\cF,\PP)$ and over a Hilbert space $\fH$ as in the previous section. Further, let $F\in L^2(\O)$. Then, as we have seen above, $F$ admits the chaos decomposition
\begin{equation}\label{eq:FChaos}
F = \sum_{q=0}^\infty I_q(h_q)
\end{equation}
with $h_0=\EE[F]$ and kernels $h_q\in\fH^{\odot q}$, $q\geq 1$.

Our aim is to measure the distance between $F$ and a centred Gaussian random variable with the same variance as $F$. We do this in terms of different probability metrics. To define them, recall that a collection $\Phi$ of measurable functions $\varphi:\RR\to\RR$ is said to be separating if for any two random variables $Y$ and $Y'$, $\EE[\varphi(Y)]=\EE[\varphi(Y')]$ for all $\varphi\in\Phi$ with $\EE[\varphi(Y)],\EE[\varphi(Y')]<\infty$ implies that $Y$ and $Y'$ are identically distributed. For such a class of functions $\Phi$ we define the probability metric $d_{\Phi}$ by putting
$$
d_{\Phi}(Y,Y'):=\sup_{\varphi\in\Phi}\big|\EE[\varphi(Y)]-\EE[\varphi(Y')]\big|\,,
$$
where $Y$ and $Y'$ are random variables satisfying $\EE[\varphi(Y)],\EE[\varphi(Y')]<\infty$ for all $\varphi\in\Phi$. Examples for such probability metrics are
\begin{itemize}
\item the total variation distance $d_{TV}:=d_{\Phi_{TV}}$, where $\Phi_{TV}=\{{\bf 1}_B:B\subset\RR\text{ a Borel set}\}$,
\item the Kolmogorov distance $d_K:=d_{\Phi_K}$, where $\Phi_K=\{{\bf 1}_{(-\infty,x]}:x\in\RR\}$,
\item the Wasserstein distance $d_W:=d_{\Phi_W}$, where $\Phi_{W}$ is the class of Lipschitz functions $\varphi:\RR\to\RR$ with $\|\varphi\|_{Lip}\leq 1$, where $\|\varphi\|_{Lip}:=\sup\{|\varphi(x)-\varphi(y)|/|x-y|:x,y\in\RR,x\neq y\}$,
\item the bounded Wasserstein distance $d_{bW}:=d_{\Phi_{bW}}$, in which case $\Phi_{bW}$ is the class of functions $\varphi:\RR\to\RR$ with $\|\varphi\|_{Lip}+\|h\|_\infty\leq 1$, where $\|\varphi\|_\infty:=\sup\{|\varphi(x)|:x\in\RR\}$.
\end{itemize}

If $F$ is as above and such that $\EE[F]=0$, and $Z\sim\cN(0,\sigma^2)$ with $\sigma^2:=\EE[F^2]$ denotes a Gaussian random variable, the main result of the seminal paper \cite{NourdinPeccati09} (see also Chapter 5 in \cite{NourdinPeccatiBook}) provides an upper bound for $d_{\Phi}(F,Z)$ by combining Stein's method for normal approximation with the Malliavin formalism as introduced in Section \ref{sec:GaussianAnalysis}. Here, $d_\Phi\in\{d_{TV},d_K,d_W,d_{bW}\}$ is one of the four probability distances introduced above. We use this bound to provide an estimate for $d_{\Phi}(F,Z)$ in terms of the kernels $h_q$ appearing in the chaotic representation \eqref{eq:FChaos} of $F$.

\begin{theorem}\label{thm:Main}
Let $F\in L^2(\O)$ be centred and such that $\EE[F^2]=\sigma^2>0$ and $F\in\DD^{1,4}$. Let $Z\sim\cN(0,\sigma^2)$ be a centred Gaussian random variable with variance $\sigma^2$ and denote by $h_q\in\fH^{\odot q}$, $q\geq 0$, the kernels in the chaotic decomposition \eqref{eq:FChaos} of $F$. Then,
\begin{equation}\label{eq:MainEstimate}
\begin{split}
d_{bW}(F,Z) &\leq d_W(F,Z) \leq \frac{c}{\sigma}\sum_{p=1}^\infty p\sum_{r=1}^{p-1} (r-1)!{p-1\choose r-1}^2\sqrt{(2(p-r))!}\,\|h_p\s_r h_p\|_{\fH^{\otimes 2(p-r)}}\\
&+\frac{c}{\sigma}\sum_{p,q=1\atop p\neq q}^\infty p\sum_{r=1}^{\min(p,q)} (r-1)!{p-1\choose r-1}{q-1\choose r-1}\sqrt{(p+q-2r)!}\,\|h_p\s_r h_q\|_{\fH^{\otimes p+q-2r}}
\end{split}
\end{equation}
with $c=\sqrt{2/\pi}$. In addition, if $F$ has a density with respect to the Lebesgue measure on $\RR$, then the same bound also holds with $c=2/\sigma$ in case of the total variation distance and $c=1/\sigma$ for the Kolmogorov distance.
\end{theorem}

\begin{remark}\rm 
The combination of Stein's method with techniques from Malliavin calculus has also been applied to functionals of Poisson random measures. In this context, a limit theorem that has the same spirit as our Theorem \ref{thm:Main} was derived in \cite{HugLastSchulte} and, in fact, it was this paper that inspired us to consider a similar question for Gaussian functionals. Also our proof follows the principal idea developed in \cite{HugLastSchulte}. However, since such functionals are much easier from a combinatorial point of view, Theorem \ref{thm:Main} has a much more neat form compared to its Poissonian analogue in \cite{HugLastSchulte}.  
\end{remark}

\begin{remark}\label{rm:inequality}\rm
By applying the Cauchy-Schwarz inequality, we have for $h_p\in \fH^{\odot p}, h_q\in\fH^{\odot q}$, $p,q\ge1$, the estimate
\begin{align} \label{eq:CS}
\|h_p\s_r h_q\|_{\fH^{\otimes p+q-2r}} 
&\le \sqrt{\|h_p \s_{p-r} h_p\|_{\fH^{\otimes 2r}}\|h_q \s_{q-r} h_q\|_{\fH^{\otimes 2r}}} \\
&\le \frac{1}{2} \big(\|h_p \s_{p-r} h_p\|_{\fH^{\otimes 2r}} + \|h_q \s_{q-r} h_q\|_{\fH^{\otimes 2r}} \big)\\
&= \frac{1}{2} \big(\|h_p \s_{r} h_p\|_{\fH^{\otimes 2(p-r)}} + \|h_q \s_{r} h_q\|_{\fH^{\otimes 2(q-r)}} \big)\,, \nonumber
\end{align}
see also \cite[Equation (6.2.4)]{NourdinPeccatiBook}. Hence, the right hand side of \eqref{eq:MainEstimate} can in principle be expressed solely in terms of the contraction norms $\|h_p\s_r h_p\|_{\fH^{\otimes 2(p-r)}}$, $p\ge1$, $r=1, \ldots, p$. However, in the course of such an approach, the term $\|h_p \s_{p} h_p\|_{\fH^{\otimes 0}} = \|h_p\|^2_{\fH^{\otimes p}}$, $p\ge1$, shows up, which in turn is not present in \eqref{eq:MainEstimate}. This term stems from the contraction norm $\|h_p\s_p h_q\|_{\fH^{\otimes q-p}}$, $p<q$, and it is precisely this term that forces us to deal with the chaotic gap arising in the context of Theorem \ref{thm:quantitative CLT} below.
\end{remark}

\begin{proof}[Proof of Theorem \ref{thm:Main}]
To simplify our presentation it is no loss of generality to assume that $\fH=L^2(A)$ for some $\sigma$-finite non-atomic measure space $(A,\cA,\mu)$. In this case, we shall write $\|\,\cdot\,\|_q$ instead of $\|\,\cdot\,\|_{\fH^{\otimes q}}$ for integers $q\geq 1$. Moreover, in order to underline that the elements of the Hilbert space we are dealing with are functions, we use the symbols $f$ and $g$ instead of $h$ and $h'$ and denote the kernels of the chaotic decomposition of $F$ by $f_q\in L^2_{\rm sym}(A^q)$, $q\geq 0$, building thereby on the notation aleady introduced in the previous section.

We prove the result only for the unit variance case $\sigma^2=1$, the general result then follows by a scaling argument exactly as in the proof of Theorem 5.1.3 in \cite{NourdinPeccatiBook}. In this set-up, the same result provides an upper bound for $d_{W}(F,Z)$, $d_{bW}(F,Z)$, $d_{K}(F,Z)$ and $d_{TV}(F,Z)$ in terms of the Malliavin operators $D$ and $L^{-1}$. 
Formally, due to the fact that $F\in\DD^{1,4}$ implies that $\int_A (D_xF)(-D_xL^{-1}F)\,\mu(\dint x) \in L^2(\Omega)$ as shown in Proposition 5.1.1 in \cite{NourdinPeccatiBook}, one has that
\be{eq:estimate}
d_{\Phi}(F,Z) \leq  c_\Phi\,\sqrt{\EE\Big[\Big(1-\int_A (D_xF)(-D_xL^{-1}F)\,\mu(\dint x\Big)^2\Big]}
\ee
with
\begin{equation}\label{eq:ConstantcH}
c_\Phi=\begin{cases}
\sqrt{2\over\pi} &: \Phi=\Phi_{W}\text{\ or \ }\Phi=\Phi_{bW}\\
1 &: \Phi=\Phi_K\\
2 &: \Phi=\Phi_{TV}\,,
\end{cases}
\end{equation}
where we implicitly used the assumption that $F$ has a density in case of the Kolmogorov and the total variation distance, see \cite{NourdinPeccatiBook}. 
Let us abbreviate the term under the above square-root by $T(F)$. Using the variance representation \eqref{eq:Variance} and the definitions \eqref{eq:DefD} and \eqref{eq:defL-1} of $D$ and $L^{-1}$, respectively, $T(F)$ can be re-written as
$$
T(F)=\EE\bigg[\Big(\int_A \sum_{p=1}^\infty pI_{p-1}(f_p(x,\,\cdot\,))\,\sum_{q=1}^\infty I_{q-1}(f_q(x,\,\cdot\,))\,\mu(\dint x)-\sum_{n=1}^\infty n!\|f_n\|_n^2\Big)^2\bigg]\,.
$$
Thus, applying the inequality $\sqrt{\var(X+Y)}\le \sqrt{\var(X)} + \sqrt{\var(Y)}$ yields that $\sqrt{T(F)}$ can be estimated from above by
\begin{align*}
&\sum_{p,q=1}^\infty p\,\Bigg(\EE\bigg[\Big(\int_A I_{p-1}(f_p(x,\,\cdot\,))I_{q-1}(f_q(x,\,\cdot\,))\,\mu(\dint x)\\
&\qquad\qquad\qquad\qquad-\EE\int_AI_{p-1}(f_p(x,\,\cdot\,))I_{q-1}(f_q(x,\,\cdot\,))\,\mu(\dint x)\Big)^2\bigg]\Bigg)^{1/2}\\
& = \sum_{p,q=1}^\infty p\,\Bigg(\Var\Big(\int_AI_{p-1}(f_p(x,\,\cdot\,))I_{q-1}(f_q(x,\,\cdot\,))\,\mu(\dint x)\Big)\Bigg)^{1/2}\,,
\end{align*}
where we used the It\^o isometry \eqref{eq:Isometry} to get an alternative expression for the term $\sum\limits_{n=1}^\infty n!\|f_n\|_n^2 $. Next, we compute the variance, using that
$$
\Var\Big(\int_AI_{p-1}(f_p(x,\,\cdot\,))I_{q-1}(f_q(x,\,\cdot\,))\,\mu(\dint x)\Big) = T_1(F)-T_2(F)^2
$$
with $T_1(F)$ and $T_2(F)$ given by
$$
T_1(F) := \int_A\int_A\EE\big[I_{p-1}(f_p(x,\,\cdot\,))I_{q-1}(f_q(x,\,\cdot\,))I_{p-1}(f_p(y,\,\cdot\,))I_{q-1}(f_q(y,\,\cdot\,))\big]\,\mu(\dint y)\mu(\dint x)
$$
and
$$
T_2(F) := \EE\int_A I_{p-1}(f_p(x,\,\cdot\,))I_{q-1}(f_q(x,\,\cdot\,))\,\mu(\dint x)\,.
$$
To compute $T_1(F)$ we use twice the multiplication formula \eqref{eq:MultiplicationFormula} together with the stochastic Fubini theorem \cite[Theorem 5.13.1]{PeccatiTaqquBook} and the isometry property \eqref{eq:Isometry}. We obtain that
\begin{align*}
T_1(F) &= \int_A\int_A\EE\Bigg[\sum_{r=0}^{\min(p-1,q-1)}\sum_{s=0}^{\min(p-1,q-1)}r!s!{p-1\choose r}{q-1\choose r}{p-1\choose s}{q-1\choose s}\\
&\qquad\times I_{p+q-2(r+1)}(f_p(x,\,\cdot\,)\s_r f_q(x,\,\cdot\,)) I_{p+q-2(s+1)}(f_p(y,\,\cdot\,)\s_s f_q(y,\,\cdot\,))\Bigg]\,\mu(\dint y)\mu(\dint x)\\
&= \sum_{r=0}^{\min(p-1,q-1)}\sum_{s=0}^{\min(p-1,q-1)}r!s!{p-1\choose r}{q-1\choose r}{p-1\choose s}{q-1\choose s}\\
&\qquad\qquad\qquad\times \EE\big[I_{p+q-2(r+1)}(f_p\s_{r+1} f_q) I_{p+q-2(s+1)}(f_p\s_{s+1} f_q)\big]\\
&=\sum_{r=1}^{\min(p,q)}((r-1)!)^2{p-1\choose r-1}^2{q-1\choose r-1}^2(p+q-2r)!\,\|f_p\ts_r f_q\|_{p+q-2r}^2\,.
\end{align*}
On the other hand, we have 
$$
T_2(F) = {\bf 1}(p=q)\,\int_A (p-1)!\|f_p(x,\,\cdot\,)\|_{p-1}^2\,\mu(\dint x) = {\bf 1}(p=q)\,(p-1)!\|f_p\|_p^2\,,
$$
which is just the square root of the last summand in the expression for $T_1(F)$ for $p=q$. Consequently, combining the expressions for $T_1(F)$ and $T_2(F)$ yields
\begin{align*}
& T_1(F)-T_2(F)^2\\
&= {\bf 1}(p=q)\sum_{r=1}^{p-1}((r-1)!)^2{p-1\choose r-1}^4(2(p-r))!\|f_p\ts_r f_p\|_{2(p-r)}^2\\
&\qquad +{\bf 1}(p\neq q)\sum_{r=1}^{\min(p,q)}((r-1)!)^2{p-1\choose r-1}^2{q-1\choose r-1}^2(p+q-2r)!\,\|f_p\ts_r f_q\|_{p+q-2r}^2\,.
\end{align*}
Together with the elementary inequality $\sqrt{a+b}\leq\sqrt{a}+\sqrt{b}$, valid for all $a,b\geq 0$, and the fact that, by Jensen's inequality, $\|\widetilde{g}\|_p\leq\|g\|_p$ for all $g\in L^2(A^p)$, $p\geq 1$, this implies that
\begin{align*}
d_{\Phi}(F,Z) &\leq c_\Phi\,\sum_{p=1}^\infty p\sum_{r=1}^{p-1} (r-1)!{p-1\choose r-1}^2\sqrt{(2(p-r))!}\,\|f_p\ts_r f_p\|_{2(p-r)}\\
&\qquad +c_\Phi\,\sum_{p,q=1\atop p\neq q}^\infty p\sum_{r=1}^{\min(p,q)}  (r-1)!{p-1\choose r-1}{q-1\choose r-1}\sqrt{(p+q-2r)!}\,\|f_p\ts_r f_q\|_{p+q-2r}\\
&\leq c_\Phi\,\sum_{p=1}^\infty p\sum_{r=1}^{p-1} (r-1)!{p-1\choose r-1}^2\sqrt{(2(p-r))!}\,\|f_p\s_r f_p\|_{2(p-r)}\\
&\qquad +c_\Phi\,\sum_{p,q=1\atop p\neq q}^\infty p\sum_{r=1}^{\min(p,q)}  (r-1)!{p-1\choose r-1}{q-1\choose r-1}\sqrt{(p+q-2r)!}\,\|f_p\s_r f_q\|_{p+q-2r}\,.
\end{align*}
The proof is thus complete.
\end{proof}

\begin{remark}\rm 
It appears that the formula for the scalar product $\int_A (D_xF)(-D_xL^{-1}F)\,\mu(\dint x)$ has already been computed in Equation (6.3.2) in \cite{NourdinPeccatiBook}. However, it has not been used in \cite{NourdinPeccatiBook} to derive a quantitative central limit theorem. 
\end{remark}

Note that the assumption $F\in\DD^{1,4}$ in Theorem \ref{thm:Main} justifies the inequality at \eqref{eq:estimate}, but does not necessarily imply that the sums on the right hand side of \eqref{eq:MainEstimate} converge. For this to hold, extra assumptions are needed. However, it turns out that they are not too restrictive in the applications we have in mind, see Theorem \ref{thm:quantitative CLT}.

Let us briefly consider two special cases, namely that $F$ belongs to a single Wiener chaos or to a finite sum of Wiener chaoses. Here, the result reduces to Proposition 3.2 or Proposition 3.7 in \cite{NourdinPeccati09}, respectively. Note that in these cases \cite[Theorem 2.10.1]{NourdinPeccatiBook} ensures that the functional $F$ has a density with respect to the Lebesgue measure on $\RR$. Moreover, one easily verifies that $F\in\DD^{1,4}$. For simplicity we decided to restrict to the unit variance case only, which is, as explained above, no restriction of generality.

\begin{corollary}
Let $Z$ be a centred Gaussian random variable with unit variance and let $d_\Phi$ be one of the probability metrics $d_{TV},d_{K},d_W$ or $d_{bW}$.
\begin{itemize}
\item[(a)] If $F=I_q(h)$ for some integer $q\ge2$ and an element $h\in\fH^{\odot q}$ such that $\EE[F^2]=1$. Then there is a constant $c_1\in(0,\infty)$ only depending on $q$ and the choice of the probability metric such that $$d_{\Phi}(F,Z)\leq c_1\,\max_{r=1,\ldots,q-1}\|h\s_r h\|_{\fH^{\otimes 2(q-r)}}\,.$$
\item[(b)] If $F=I_{q_1}(h_1)+\ldots+I_{q_n}(h_n)$ for integers $n,q_1,\ldots,q_n\geq 1$ and elements $h_i\in\fH^{\odot q_i}$, $i=1,\ldots,n$, such that $\EE[F^2]=1$. Then there are constants $c_1,c_2\in(0,\infty)$ only depending on $q_1,\ldots,q_n$ and on the choice of the probability metric such that $$d_{\Phi}(F,Z)\leq c_1\,\max_{r=1,\ldots,q_i-1\atop i=1,\ldots,n}\|h_i\s_r h_i\|+c_2\,\max_{r=1,\ldots,\min(q_i,q_j)\atop 1\leq i<j\leq n}\|h_i\s_r h_j\|\,.$$
\end{itemize} 
\end{corollary}

\section{Application to Gaussian processes}\label{sec:Application}

\subsection{The Breuer-Major Theorem}
Let $(a_n)_{n\in\NN}$, $(b_n)_{n\in\NN}$ be two positive sequences. Then we write $a_n\lesssim b_n$ if $a_n/b_n$ is bounded, and $a_n \sim b_n$ whenever $a_n\lesssim b_n$ and $b_n \lesssim a_n$. 

Consider a one-dimensional centred and stationary Gaussian process $X = (X_k)_{k\in\ZZ}$ with unit variance and a covariance function
\be{eq:covariance}
\r(j) = \EE[X_1X_{1+j}]\,,\qquad j\in\ZZ\,.
\ee
In what follows we will assume that
\begin{equation}\label{eq:AssumtionsGP}
|\r(j)| \sim |j|^{-\a}
\end{equation}
for some $\a>0$. Recall that the Cauchy-Schwarz implies that $|\r(j)|\le \r(0)=1$ for all $j\in\ZZ$. For technical reasons, we assume that for any $n\ge1$ the vector $(X_1, \ldots, X_n)$ is jointly Gaussian with a non-degenerate covariance matrix. As an example, one can think of $X$ being obtained from the increments of a fractional Brownian motion $B^H=(B_t^H)_{t\in\RR}$ with Hurst parameter $H\in(0,1)$, that is, $X_k=B_{k+1}^H-B_k^H$ for all $k\in\ZZ$. In that case
$$
\r(j)={1\over 2}\big(|j+1|^{2H}+|j-1|^{2H}-2|j|^{2H}\big)
$$
and thus \eqref{eq:AssumtionsGP} is satisfied with $\a=2-2H$, see Chapter 7.4 in \cite{NourdinPeccatiBook}.

Let $g\colon \RR\to\RR$ be a non-constant measurable function with $\EE|g(X_1)|^2<\infty$ and consider the partial sum
\be{eq:F_n}
F_n = \frac{1}{\sqrt{n}}\sum_{k=1}^n \{ g(X_k) - \EE[g(X_k)]\}\,,\qquad n\in\NN\,.
\ee
We will assume that $g$ has {Hermite rank} equal to $m\in\NN$. That is, for each polynomial $p\colon\RR\to\RR$ with degree $\in\{1,\ldots,m-1\}$ one has that $\EE[(g(X_1) - \EE[g(X_1)])p(X_1)]=0$ and $\EE[g(X_1)H_m(X_1)]\neq0$, where $H_m$ is the $m$th Hermite polynomial. The Hermite rank can also be described in a different manner. Namely, due to the moment condition $\EE|g(X_1)|^2<\infty$, $g$ has the following Hermite expansion:
\be{eq:Hermite expansion0}
g(x) - \EE[g(X_1)] = \sum_{q=1}^\infty c_qH_q(x)\qquad\text{with}\qquad c_q = \frac{1}{q!} \EE[g(X_1) H_q(X_1)]\,.
\ee
For $g$ to have Hermite rank $m$ means that $c_q=0$ for all $q\in\{1,\ldots,m-1\}$ and that $c_m\neq0$. (The existence of such an $m$ is implied by the fact that $g$ is non-constant.)

It turns out that the rates in Theorem \ref{thm:quantitative CLT} below do not only depend on $\a$ and the Hermite rank $m$, but also on a quantity $\gamma\in\NN\cup\{\infty\}$ that we call the \emph{chaotic gap} of $g$. Roughly speaking, it is the minimal distance between two active chaoses a functional lives in. More precisely, if $c_q=0$ for all $q\neq m$, we set $\gamma=\infty$. If there is a $p\ge1$ such that $c_p\neq0$ and $c_{p+1}\neq 0$, we set $\g=1$. Otherwise, $\gamma\ge2$ and it is uniquely characterized by the following two conditions:
\begin{itemize}
\item[(i)] for all $q\geq 1$: if $c_q\neq 0$ then $c_{q+1} = \cdots = c_{q+\g-1}=0$,
\item[(ii)] there exists $p\geq 1$ such that $c_p\neq 0$ and $c_{p+\g}\neq 0$.
\end{itemize}
Of course, if $g= H_q$ for some $q\ge1$, then $m=q$ and $\gamma=\infty$. Also if $g$ is a linear combination of Hermite polynomials, one can directly determine the rank and the chaotic gap. More general examples involve even and odd functions, both having a chaotic gap of $\gamma=2$, and the exponential function having a chaotic gap of $\gamma=1$. Another interesting example is the indicator function $g=\one_{[0,\infty)}$, which satisfies $m=1$ and $\gamma=2$.

Using the parameters $\a,\g$ and $m$ we can now formulate our quantitative central limit theorem for the random variables $F_n$ defined at \eqref{eq:F_n}.

\begin{theorem}\label{thm:quantitative CLT}
Let $g\colon\RR\to\RR$ be a measurable and non-constant function. Suppose that $\EE|g(X_1)|^2<\infty$, $\EE|g((2+\varepsilon)X_1)|^2<\infty$ for some $\varepsilon>0$, that $g$ has Hermite rank equal to $m\ge2$, that $\r(0) = 1$ and that $|\r(j)| \sim  |j|^{-\a}$ for some $\a >1/m$. Further, denote by $\g\in\NN\cup\{\infty\}$ the chaotic gap of $g$ and let $Z$ be a standard Gaussian variable. Then the following assertions are true for each of the probability metrics $d_\Phi\in\{d_W,d_{bW}\}$ and also for $d_{TV}$ and $d_K$ in the case that $F_n$ has a density with respect to the Lebesgue measure on $\RR$.
\begin{enumerate}
\item
If $m=2$ and $\g=1$ it holds that
\be{eq:CLT for F_n a}
d_\Phi\left(\frac{F_n}{\sqrt{\Var(F_n)}}, Z\right) 
\lesssim
\frac{1}{\Var(F_n)}\times 
\begin{cases}
n^{-1/2} &: \a>1  \\ 
n^{-\a/2} &: \a\in\Big(\frac{2}{3}, 1\Big)\\ 
n^{1-2\a} &: \a\in\Big(\frac{1}{2},\frac{2}{3}\Big]\,.
\end{cases}
\ee
\item
If $m=2$ and $\g\ge2$ it holds that
\be{eq:CLT for F_n b}
d_\Phi\left(\frac{F_n}{\sqrt{\Var(F_n)}}, Z\right) 
\lesssim 
\frac{1}{\Var(F_n)}\times 
\begin{cases}
n^{-1/2} &: \a>\frac{3}{4}  \\ 
n^{1-2\a} &: \a\in\Big(\frac{1}{2},\frac{3}{4}\Big)\,.
\end{cases}
\ee
\item
If $m\ge3$ and $\g=1$ it holds that
\be{eq:CLT for F_n c}
d_\Phi\left(\frac{F_n}{\sqrt{\Var(F_n)}}, Z\right) 
\lesssim
\frac{1}{\Var(F_n)}\times 
\begin{cases}
n^{-1/2} &: \a>1  \\ 
n^{-\a/2} &: \a\in\Big(\frac{1}{m-\frac12}, 1\Big)\\ 
n^{1-m\a} &: \a\in\Big(\frac{1}{m},\frac{1}{m-\frac12}\Big]\,.
\end{cases}
\ee
\item
If $m\ge3$ and $\g\ge2$ it holds that
\be{eq:CLT for F_n c}
d_\Phi\left(\frac{F_n}{\sqrt{\Var(F_n)}}, Z\right) 
\lesssim
\frac{1}{\Var(F_n)}\times 
\begin{cases}
n^{-1/2} &: \a>\frac12  \\ 
n^{-\a} &: \a\in\Big(\frac{1}{m-1}, \frac12\Big)\\ 
n^{1-m\a} &: \a\in\Big(\frac{1}{m},\frac{1}{m-1}\Big)\,.
\end{cases}
\ee
\end{enumerate}
Moreover, as $n\to\infty$, one has that $$\Var(F_n) \to \sigma^2:= \Var(g(X_1)) + 2\sum_{k=1}^\infty \Cov[g(X_1), g(X_{1+k})]\in[0,\infty)\,.$$
\end{theorem}

\begin{remark}\rm \label{rem:ExceptionalCases}
In Theorem \ref{thm:quantitative CLT} we have excluded some boundary cases, for example in part (b) the case that $\a=3/4$. It is possible to fill these gaps and to derive rates of convergence, which involve logarithmic terms. For sake of simplicity we have excluded them from our discussion.
\end{remark}

Let us briefly comment on the assumptions made in Theorem \ref{thm:quantitative CLT}. At first, one might wonder whether the condition $\EE|g((2+\varepsilon)X_1)|^2<\infty$ for some $\e>0$ is already implied by the condition that $\EE|g(X_1)|^2<\infty$. Whilst this is true for many choices of $g$ such as $g(x)=|x|^p$, this is not generally the case as the following example shows. Due to our assumptions, $X_1\sim\cN(0,1)$ has a standard Gaussian distribution and we observe that 
$$
M(a):=\EE\big[e^{aX_1^2}\big] = {1\over\sqrt{2\pi}}\int_{-\infty}^\infty e^{(a-{1\over 2})x^2}\,\dint x
= \begin{cases}
 {1\over\sqrt{1-2a}} &: a<{1\over 2} \\
\infty &: a\geq{1\over 2}\,.
\end{cases}
$$
Thus, taking $g(x):=e^{x^2/8}$ we conclude that $\EE|g(X_1)|^2=M(1/4)=\sqrt{2}$, while $\EE|g((2+\varepsilon)X_1)|^2=M(1 + \varepsilon +\varepsilon^2/4)=\infty$ for all $\varepsilon>0$.
Moreover, the motivation to impose the moment condition $\EE|g((2+\varepsilon)X_1)|^2<\infty$ for some $\e>0$ is to ensure the convergence of the corresponding sums in \eqref{eq:MainEstimate}, thus also implying that $\langle D F_n, -DLF_n\rangle_{\cH}\in L^2(\O)$ such that one can formally apply Theorem \ref{thm:Main}. As a consequence and in contrast to Section 6 in \cite{NourdinPeccatiReinertSecondOrderPoincare}, we do not need moment assumptions involving derivatives of $g$ such as $\EE|g'(X_1)|^4<\infty$, which would in turn imply that $F_n\in \DD^{1,4}$. Hence, we can dispense with additional smoothness or regularity assumptions on $g$ and are able to handle even non-continuous choices of $g$.

Theorem \ref{thm:quantitative CLT} also raises the question under which conditions the partial sums $F_n$ have a density with respect to the Lebesgue measure on the real line. 
To give an answer to this question which goes beyond the (somehow restrictive) conditions of the transformation theorem for densities, we introduce the notion of a $0$-measure-preserving map. A measurable function $f\colon \RR\to\RR$ is called $0$-measure-preserving if for all Lebesgue null sets $B\subset\RR$ also the preimage $f^{-1}(B)\subset\RR$ is a Lebesgue null set. 
Using the technical assumption that $(X_1, \ldots, X_n)$ is jointly Gaussian with a non-degenerate covariance matrix, it turns out that $F_n$ has a density with respect to the Lebesgue measure on $\RR$ if and only if $g$ is $0$-measure-preserving. An argument is given in the Appendix. Examples of functions $g$ where $F_n$ has no density are given by locally constant functions, including, for example, the case of indicator functions. In turn, the power function $g(x)=|x|^p$ is $0$-measure-preserving such that the functionals considered in Subsection \ref{subsec:power variation} below possess a density for any choice of $p>0$.

\medskip

Some parts of the proof that involve convolutions of sequences with finite support and require an application of Young's inequality are inspired by \cite{BiermeBonamiLeon2011}. We write $\ell^p(\ZZ)$ for the space of sequences $u=(u_k)_{k\in\ZZ}$ such that $\|u\|_{\ell^p(\ZZ)}:=\big(\sum_{k\in\ZZ}|u_k|^p\big)^{1/p}<\infty$ if $p\in[1,\infty)$ and $\|u\|_{\ell^\infty(\ZZ)}:=\sup_{k\in\ZZ}|u_k|<\infty$ if $p=\infty$. Now, recall that the convolution of two sequences $u,v$ on $\ZZ$ with finite support is defined as 
\[
(u * v) (k) := \sum_{j\in\ZZ} u(j) v(k-j)\,, \qquad k\in\ZZ\,,
\]
and that $u * v$ has again a finite support. Due to Young's inequality, it holds that for $1\le p,q,r\le \infty$ with $1/p + 1/q = 1+1/r$,
\be{eq:Young}
\|u * v\|_{\ell^r(\ZZ)}\le \|u\|_{\ell^p(\ZZ)} \, \|v\|_{\ell^q(\ZZ)}
\ee
for sequences $u$ and $v$ with finite support.

\begin{proof}[Proof of Theorem \ref{thm:quantitative CLT}]
It is no loss of generality to assume that $\EE[g(X_1)]=0$. Then, since $g$ is assumed to have Hermite rank equal to $m$, we have the unique Hermite expansion
\be{eq:Hermite expansion}
g(x) = \sum_{q=m}^\infty c_qH_q(x)\qquad\text{with}\qquad c_q = \frac{1}{q!} \EE[g(X_1) H_q(X_1)]\,,
\ee
where the sum converges in the $L^2$-sense, meaning that $\EE|g(X_1) - \sum_{q=m}^n c_qH_q(X_1)|^2\to0$, as $n\to\infty$.
Moreover due to our assumption that $\EE|g((2+\varepsilon)X_1)|^2<\infty$ and thanks to \eqref{eq:Hermite moments}, we have that
\be{eq:summability}
\sum_{q=m}^\infty q! \,c_q^2\,(2+\varepsilon)^{2q}<\infty\,,
\ee
see e.g.~\cite[Proposition 1.4.2]{NourdinPeccatiBook}.

The next is to observe that one can consider the Gaussian process $X = (X_k)_{k\in \ZZ}$ as a subset of an isonormal Gaussian process $\{W(h): h\in \H\}$, say, where $\H$ is a real separable Hilbert space with scalar product $\langle\, \cdot\,,\, \cdot\,\rangle_{\H}$. This means that for every $k\in\ZZ$ there exists an element $h_k\in\H$ such that $X_k =W(h_k)$ and, consequently,
\be{} 
\langle h_k, h_\ell \rangle_\H = \r(k-\ell)\qquad\text{for all}\qquad k,\ell\in\ZZ\,,
\ee
see \cite[Proposition 7.2.3]{NourdinPeccatiBook} for details.
Using that for $h\in\fH$, $I_q(h^{\otimes q}) = H_q(W(h))$ ($H_q$ is again the $q$th Hermite polynomial), we see that $g(X_k) = \sum_{q=m}^\infty c_q I_q(h_k^{\otimes q})$ and hence
\be{}
F_n = \sum_{q=m}^\infty I_q(f_{q,n})\qquad\text{with}\qquad f_{q,n} = \frac{c_q}{\sqrt{n}} \sum_{k=1}^n h_k^{\otimes q} \in \H^{\odot q}\,.
\ee
For $p,q\ge m$ and $r=1, \ldots, \min(p,q)$ we compute
\begin{align*}
\|f_{q,n} \|^2_{\H^{\otimes q}} &= \frac{c_q^2}{n}\sum_{k,\ell=1}^n\r^q(k-\ell)\,, \\
f_{p,n} \otimes_r f_{q,n} &= \frac{c_p\,c_q}{n}\sum_{k,\ell=1}^n\r^r(k-\ell) \, [h_k^{\otimes p-r} \otimes h_\ell^{\otimes q-r}]\,, \\
\|f_{p,n} \otimes_r f_{q,n}  \|^2_{\H^{\otimes p+q-2r}} &=\frac{c_p^2\,c_q^2}{n^2}\sum_{i,j,k,\ell=1}^n\r^r(i-j)\r^r(k-\ell) \,\r^{p-r}(i-k)\r^{q-r}(j-\ell) \,. 
\end{align*}
Since $\sum_{k\in\ZZ} |\r(k)|^q<\infty$ for all $q\ge m$, one has by dominated convergence that
\[
\|f_{q,n} \|^2_{\H^{\otimes q}} = c_q^2\sum_{k=-(n-1)}^{n-1}\frac{n-|k|}{n} \r^q(k)
\to c_q^2 \sum_{k\in\ZZ} \r^q(k) \le c_q^2 \sum_{k\in\ZZ}|\r(k)|^m<\infty\,,
\]
as $n\to\infty$. With respect to the summability condition \eqref{eq:summability} and the variance representation \eqref{eq:Variance}, this implies that, as $n\to\infty$,
\be{eq:Varianz}
\sigma_n^2:=\Var(F_n) = \sum_{q=m}^\infty q!\|f_{q,n}\|^2\to\sigma^2 := \sum_{q=m}^\infty q!\, c_q^2 \sum_{k\in\ZZ} \r^q(k)  \in[0,\infty)\,.
\ee
In view of our main bound \eqref{eq:MainEstimate} we need to compute the asymptotic order of the quantity
\be{eq:A_n}
A_n(p,q,r):= {1\over n^2}\sum_{i,j,k,\ell=1}^n|\r(i-j)|^r |\r(k-\ell)|^r |\r(i-k)|^{p-r} |\r(j-\ell)|^{q-r}\,,\quad p,q\ge m
\ee
for $r=1, \ldots, \min(p,q)$ if $p\neq q$ and for $r=1, \ldots, q-1$ if $p=q$. 
We assume without loss of generality that $p\le q$ and distinguish two cases. First, let $r=1, \ldots, p-1$. Then
\[
A_n(p,q,r)\le A_n(p,p,r) = A_n(p,p,p-r) \le A_n(m,m,\min(r,p-r, m-1))\,.
\]
Second, assume that $m\le r =p<q$ such that there exists some integer $t\ge1$ with $p+t=q$. Note that we only have to consider those cases where $t$ is greater than or equal to the chaotic gap $\g$. Then
\[
A_n(p,q,p)\le A_n(m,m+t,m)\le A_n(m,m+\g,m)\,.
\]
By index shifting, one obtains
\begin{align*}
A_n(p,q,r)&= {1\over n^2} \sum_{i,j,k,\ell=0}^{n-1}|\r(|i-j|)|^r |\r(|k-\ell|)|^r |\r(|i-k|)|^{p-r} |\r(|j-\ell|)|^{q-r} \\
&={1\over n^2} \sum_{\ell=0}^{n-1} \,\sum_{j=-l}^{n-1-\ell}\, \sum_{i=-j}^{n-1-j}\,\sum_{k=-\ell}^{n-1-\ell}
|\r(|i|)|^r |\r(|k|)|^r |\r(|i+j-k|)|^{p-r} |\r(|j|)|^{q-r} \\
& \le {1\over n} \sum_{|j|\le n-1} \, \sum_{i=-j}^{n-1-j} \, \sum_{|k|\le n-1} 
|\r(|i|)|^r |\r(|k|)|^r |\r(|k- (i+j)|)|^{p-r} |\r(|j|)|^{q-r} \,.
\end{align*}
This means that in the second case we get
\be{eq:rate with e}
\begin{split}
A_n(p,p+t,p) \le A_n(m,m+\g,m)&\le {1\over n}\left(\,\sum_{k\in\ZZ} |\r(k)|^m\right)^2\sum_{|j|\le n-1} |\r(j)|^\g\\
& \lesssim
\begin{cases}
n^{-1} &: \a>\frac{1}{\g} \\
n^{-\a\g} &: \a<\frac{1}{\g}\,.
\end{cases}
\end{split}
\ee
Now, let us come back to the first case and let $r = 1, \ldots, m-1$. For any integer $s\ge1$ introduce the truncated sequence
\be{eq:truncation}
\r_{s,n}(k) = |\r(k)|^s\, \one(|k|\le n-1)\,.
\ee
Then, using again a careful index shifting, we see that
\begin{align*}
A_n(m,m,r) 
&= {1\over n^2}\sum_{i,j,k,\ell=0}^{n-1}\r_{r,n}(|i-j|) \r_{r,n}(|k-\ell|) \r_{m-r,n}(|i-k|) \r_{m-r,n}(|j-\ell|) \\
&\le {1\over n^2} \sum_{i,j=0}^{n-1} (\r_{r,n} * \r_{m-r,n}) (|i-j|)^2 \\
&\le {1\over n} \sum_{|j|\le n-1} (\r_{r,n} * \r_{m-r,n}) (j)^2 \\
&\le {1\over n}\, \|\r_{r,n} * \r_{m-r,n}\|^2_{\ell^2(\ZZ)}\,.
\end{align*}
Now, we apply Young's inequality \eqref{eq:Young} to derive a rate for $\|\r_{r,n} * \r_{m-r,n}\|^2_{\ell^2(\ZZ)}$. For $m=2$ it holds that 
\[
\|\r_{1,n} * \r_{1,n}\|^2_{\ell^2(\ZZ)}\le \|\r_{1,n} \|^4_{\ell^{4/3}(\ZZ)} = \left(\sum_{|k|\le n-1}|\r(k)|^{4/3}\right)^3\lesssim 
\begin{cases}
1 &: \a>\frac34\\
n^{3-4\a} &: \a< \frac{3}{4}
\end{cases}
\]
and if $m>2$ we find that
\[
\|\r_{r,n} * \r_{m-r,n}\|^2_{\ell^2(\ZZ)}\le \|\r_{r,n} \|^2_{\ell^2(\ZZ)}\|\r_{m-r,n} \|^2_{\ell^1(\ZZ)}\,.
\]
Moreover,
\begin{align}\label{eq:rate1}
\|\r_{r,n} \|^2_{\ell^2(\ZZ)} &= \sum_{|k|\le n-1} |\r(k)|^{2r} \lesssim
\begin{cases}
1 &: \a>\frac{1}{2r}\\
n^{1-2r\a} &: \a<\frac{1}{2r}\,,
\end{cases} \\ 
\label{eq:rate2}
\|\r_{m-r,n} \|^2_{\ell^1(\ZZ)} &= \left(\sum_{|k|\le n-1} |\r(k)|^{m-r}\right)^2 \lesssim 
\begin{cases}
1 &: \a>\frac{1}{m-r}\\
n^{2-2(m-r)\a} &: \a<\frac{1}{m-r}\,.
\end{cases}
\end{align}
Consequently, by changing the roles of $\r_{r,n} $ and $\r_{m-r,n}$, one has the estimate 
\begin{align*}
\|\r_{r,n} * \r_{m-r,n}\|^2_{\ell^2(\ZZ)}
&\le \max_{r=1, \ldots, [(m-1)/2]} \|\r_{r,n} \|^2_{\ell^2(\ZZ)}\|\r_{m-r,n} \|^2_{\ell^1(\ZZ)} \\
&\lesssim
\begin{cases}
1 &: \a>\frac12\\
n^{1-2\a} &: \a\in\left(\frac{1}{m-1}, \frac12\right) \\
n^{3-2m\a} &: \a<\frac{1}{m-1} \,.
\end{cases}
\end{align*}
In summary, we arrive at the bounds
\begin{align}
A_n(2,2,1)
&\lesssim
\begin{cases}
n^{-1} &: \a>\frac34\\
n^{2-4\a} &: \a< \frac{3}{4}\,,
\end{cases}\\[1em]
A_n(m,m,r)&\lesssim
\begin{cases}
n^{-1} &: \a>\frac12\\
n^{-2\a} &: \a\in\left(\frac{1}{m-1}, \frac12\right) \\
n^{2-2m\a} &: \a<\frac{1}{m-1} \,.
\end{cases}
\end{align}

Now, one can plug-in the estimates into the right hand side of \eqref{eq:MainEstimate} and obtain 
by defining $A_n:=\big(\max_{r=1,\ldots,[(m-1)/2]}A_n(m,m,r) + A_n(m, m+\g, m)\big)^{1/2}$ that
\begin{align}\nonumber
&\sum_{p=m}^\infty p\sum_{r=1}^{p-1} (r-1)!{p-1\choose r-1}^2\sqrt{(2(p-r))!}\,\Big\|\tfrac{f_{p,n}}{\sigma_n}\s_r \tfrac{f_{p,n}}{\sigma_n}\,\Big\|_{2(p-r)}\\ \nonumber
&\qquad + \sum_{p,q=m\atop p\neq q}^\infty p\sum_{r=1}^{\min(p,q)}  (r-1)!{p-1\choose r-1}{q-1\choose r-1}\sqrt{(p+q-2r)!}\,\Big\|\tfrac{f_{p,n}}{\sigma_n}\s_r \tfrac{f_{q,n}}{\sigma_n}\,\Big\|_{p+q-2r} \\ \nonumber
&\le \frac{A_n}{\sigma_n^2}
\bigg\{ \sum_{p=m}^\infty c_p^2\, p \sum_{r=1}^{p-1} (r-1)!{p-1\choose r-1}^2\sqrt{(2(p-r))!}\\
&\qquad + \sum_{p,q=m\atop p\neq q}^\infty |c_p| |c_q|\,p \sum_{r=1}^{\min(p,q)}  (r-1)!{p-1\choose r-1}{q-1\choose r-1}\sqrt{(p+q-2r)!}\bigg\} \label{eq:two sums}\,.
\end{align}
The claim follows upon proving that the two sums in the brackets converge. Put
\be{eq:B_1}
B_1(p):= \sum_{r=1}^{p-1} (r-1)!{p-1\choose r-1}^2\sqrt{(2(p-r))!}\,,
\ee
note that due to \eqref{eq:summability}, the first sum converges provided that $B_1(p) \lesssim (p-1)!\,4^{p}\lesssim (p-1)!\,(2+\varepsilon)^{2p}$.
However, using the Cauchy-Schwarz inequality, Vandermonde's identity for binomial coefficients and Stirling's formula, we see that
\begin{align*}
\frac{B_1(p)}{(p-1)!} &= \sum_{r=1}^{p-1}{p-1\choose r-1}{2(p-r)\choose p-r}^{1/2}
\le \left(\sum_{r=1}^{p-1}{p-1\choose r-1}^2\right)^{1/2} \, \left(\sum_{r=1}^{p-1}{2(p-r)\choose p-r}\right)^{1/2} \\
&= \left({2(p-1)\choose p-1} - 1\right)^{1/2} \, \left(\sum_{r=1}^{p-1}{2r\choose r}\right)^{1/2}\\
&\le (p-1)^{1/2}\, {2(p-1)\choose p-1} 
\sim \frac{4^{p-1}}{\sqrt{\pi}}\,.
\end{align*}
To show that the sum in \eqref{eq:two sums} converges, it is sufficient to prove that 
\be{eq:sum convergence}
\sum_{p,q=m\atop p\neq q}^\infty (c_p^2 + c_q^2)\,p \sum_{r=1}^{\min(p,q)}  (r-1)!{p-1\choose r-1}{q-1\choose r-1}\sqrt{(p+q-2r)!}<\infty\,,
\ee
thanks to the inequality $ab\le a^2 + b^2$, valid for all $a,b\in\RR$.
To this end, observe that for $p\ge m+1$, again by Stirling's formula,
\begin{align*}
& \sum_{q=m}^{p-1} \sum_{r=1}^{q}  (r-1)!{p-1\choose r-1}{q-1\choose r-1}\sqrt{(p+q-2r)!} \\
&\le (p-1)!\sum_{q=m}^{p-1} \sum_{r=1}^{q}  {q-1\choose r-1}{2(p-r)\choose p-r}^{1/2} \\
&\le (p-1)! \sum_{q=m}^{p-1}   {2(q-1)\choose q-1}^{1/2} q^{1/2}\,{2(p-1)\choose p-1}^{1/2} \\
&\le (p-1)!\, (p-1)^{3/2}\,{2(p-1)\choose p-1} \\
&\sim (p-1)!\, (p-1) \frac{4^{p-1}}{\sqrt{\pi}} \\
&\lesssim (p-1)! \frac{(2+\varepsilon)^{2p-2}}{\sqrt{\pi}}\,,
\end{align*}
for any $\varepsilon>0$,
which implies \eqref{eq:sum convergence} in view of \eqref{eq:summability}. 
Thus, one can formally apply Theorem \ref{thm:Main} such that there is a constant $C\in(0,\infty)$ only depending on the class $\Phi$ and $g$ (or more specifically on the sequence $(c_q)_{q\in\NN}$) such that 
\[
d_\Phi\left(\frac{F_n}{\sigma_n}, Z\right)\le C\,\frac{A_n}{\sigma_n^2}\,.
\]
The proof is complete.
\end{proof}

\begin{remark}\rm \label{rem:positive variance}
If $c_{2q}\neq0$ for some $q\ge1$, it follows immediately that $\|f_{2q,n}\|^2\to c_{2q}^2 \sum_{k\in\ZZ} \r^{2q}(k)>0$, as $n\to\infty$, and hence $\sigma^2>0$. As a consequence, we see that in this situation, the variance of $F_n$ has no influence on the rates in Theorem \ref{thm:quantitative CLT}.
\end{remark}

\begin{remark}\rm \label{rem:chaotic gap}
The chaotic gap $\g$ of the function $g$ is visible in Theorem \ref{thm:quantitative CLT} only in the case $\g=1$. As our proof shows, this is just a coincidence, since for $\g=2$ the terms involving the chaotic gap are of the same size as the other leading terms and for $\g>2$ become even subdominant in our situation. Consequently, for $\g\ge2$ one gets exactly the same rates as for $\g=\infty$, and the rates in Theorem \ref{thm:quantitative CLT} coincide with the rates given in \cite[Proposition 2.15]{BiermeBonamiLeon2011}.
\end{remark}

Let us finally in this section consider the set-up of Theorem \ref{thm:quantitative CLT} in the special case that $g$ has Hermite rank $m=1$ (which has been excluded in the statement). Clearly, if $g$ is linear (so $m=1$ and $\g=\infty$), $F_n$ is already centred Gaussian and $F_n/\sqrt{\Var(F_n)}$ coincides in distribution with the standard Gaussian random variable $Z$. However, an inspection of the proof of Theorem \ref{thm:quantitative CLT} shows that for non-linear $g$ with $m=1$ and arbitrary chaotic gap $1\le\g<\infty$ the leading term in \eqref{eq:MainEstimate} is $\|f_{1,n} \otimes_1 f_{1+\g,n}\|$, which finally yields a rate of convergence of order $n^{-1/2}$ as long as $\a>1=1/m$.

\subsection{Power variations of the fractional Brownian motion and processes from the Cauchy class}\label{subsec:power variation}

We build on the example we have seen in the previous section and let again $X=(X_k)_{k\in\ZZ}$ be a one-dimensional centred and stationary Gaussian process with unit variance and with covariance function $\r$ such that the assumption \eqref{eq:AssumtionsGP} is satisfied. Moreover, we assume again that for any $n\ge1$ the vector $(X_1, \ldots, X_n)$ is jointly Gaussian with a non-degenerate covariance matrix. As function $g$ we take now
\be{eq:power function}
g(x) := |x|^p - \mu_p\,, \quad p>0\,, \qquad \text{where}\qquad \mu_p=\EE|X_1|^p\,.
\ee
In this situation, our random variable $F_n$ defined by \eqref{eq:F_n} becomes a so-called (centred) power variation of the Gaussian process $X$. The asymptotic behaviour of these functionals has attracted considerable interest in probability theory and mathematical statistics, see \cite{BarndorffNielsenEtAl1,BiermeBonamiLeon2011}, for example, as well as the references cited therein.

Now, we notice that unless in the special case $p=2$ of the quadratic variation, the Hermite expansion of the function $g$ at \eqref{eq:power function} is not finite. Moreover, since $g$ is an even function it holds that the Hermite rank of $g$ is $m=2$ and, if $p\neq2$, that the chaotic gap is $\g=2$. Consequently, part (b) of Theorem \ref{thm:quantitative CLT} applies. (Note also that $\EE|g(X_1)|^2 = \mu_{2p}<\infty$ and $\EE|g((2+\varepsilon)X_1)|^2 <\infty$.)

Instead of repeating the bounds, let us consider a more concrete situation. We assume that the Gaussian process $X=(X_k)_{k\in\ZZ}$ is obtained from the increments of a fractional Brownian motion $B^H=(B_t^H)_{t\in\RR}$ with Hurst parameter $H\in(0,1)$. Let us recall that this means that $B^H$ is a centred Gaussian process in continuous time with covariance function given by 
$$
\EE[B_s^HB_t^H] = {1\over 2}\big(|s|^{2H}+|t|^{2H}-|s-t|^{2H}\big)\,,\qquad s,t\in\RR\,,
$$
and that $X_k=B_{k+1}^H-B_k^H$ for all $k\in\ZZ$. Then $X$ has covariance function
$$
\r(j)={1\over 2}\big(|j+1|^{2H}+|j-1|^{2H}-2|j|^{2H}\big)
$$
and \eqref{eq:AssumtionsGP} is satisfied with $\a=2-2H$. Moreover, since $(X_1, \ldots, X_n)$ possesses a non-degenerate covariance function,
$$
F_n = {1\over\sqrt{n}}\sum_{k=1}^n\{|X_k|^p-\mu_p\}
$$
has a density and moreover, it is known that $\var(F_n)$ converges, as $n\to\infty$, to a positive and finite constant; see Remark \ref{rem:positive variance}. Thus, the following result is a direct consequence of Theorem \ref{thm:quantitative CLT}.

\begin{corollary}
Let $d_\Phi$ be one of the probability distances $d_{TV},d_K,d_W$ or $d_{bW}$. Then
$$
d_\Phi\Bigg({F_n\over\sqrt{\var(F_n)}}\Bigg) \lesssim \begin{cases}
n^{-1/2} &: H\in\big(0,{5\over 8}\big) \\
n^{4H-3} &: H\in\big({5\over 8},{3\over 4}\big)\,.
\end{cases}
$$
\end{corollary}

To the best of our knowledge, the previous corollary is the only known result showing that for general power variations with $p>0$ the speed of convergence in the central limit theorem is universal. By this we mean that the rate we obtain for general $p>0$ coincides with the known rate for the quadratic variation functional, where $p=2$. This should also be compared with the discussion in Remark \ref{rem:chaotic gap} and especially with the result in \cite[Proposition 2.15]{BiermeBonamiLeon2011}.

\begin{remark}\rm 
We emphasize that for $H>3/4$ the fluctuations of $F_n$ are no more Gaussian, while for the boundary case $H=3/4$ one can derive a logarithmic rate of convergence. However and as already discussed in Remark \ref{rem:ExceptionalCases}, we will not pursue this direction in the present paper.
\end{remark}

Another flexible and prominent class of random processes $X=(X_k)_{k\in\ZZ}$ to which our theory applies are the members of the so-called Cauchy class. These processes are centred Gaussian and their covariance function is given by
$$
\r(j) = (1+|j|^\beta)^{-\alpha/\beta}\,,\qquad j\in\ZZ\,,
$$
where the parameters $\alpha$ and $\beta$ have to satisfy $\alpha>0$ and $\beta\in(0,2]$, see \cite{BarndorffNielsenEtAl1}. It is clear that these processes satisfy the assumptions made in the previous secion and that Theorem \ref{thm:quantitative CLT} can be applied. This shows that, if $\alpha>1/2$, the normalized power variations $F_n$ with parameter $p>0$ satisfy the quantitative central limit theorem
$$
d_{\Phi}\Bigg({F_n\over\sqrt{\var(F_n)}}\Bigg)\lesssim\begin{cases}
n^{-1/2} &: \alpha>{3\over 4}\\
n^{1-2\alpha} &: \alpha\in\big({1\over 2},{3\over 4}\big)\,,
\end{cases}
$$
where $d_\Phi$ can be any of the probability distances $d_{TV},d_K,d_W$ or $d_{bW}$, adding thereby to the limit theorems developed in \cite{BarndorffNielsenEtAl1}. Again, the rate of convergence is universal and does not depend on the choice of the power $p$.

\subsection{Functionals of Gaussian subordinated processes in continuous time revisited}

Finally, we apply our methods to functionals of Gaussian subordinated processes in continuous time. More precisely, we revisit the example from Section 6 in \cite{NourdinPeccatiReinertSecondOrderPoincare} and show that the rate of convergence there can be improved by our methods, confirming thereby the conjecture made in Remark 6.2 in \cite{NourdinPeccatiReinertSecondOrderPoincare}. That is, we consider a centred Gaussian process $X= (X_t)_{t\in\RR}$ in continuous time with stationary increments. For example, $X$ could be a (two-sided) fractional Brownian motion. The covariance function of $X$ is defined by $\rho(u-v):= \EE[(X_{u+1} - X_u)(X_{v+1} - X_v)]$, $u,v\in\RR$. It is clear that $X$ might be considered as a suitable isonormal Gaussian process, see Example 2.1.5 in \cite{NourdinPeccatiBook}.

Let $g\colon \RR\to\RR$ be a non-constant and measurable function and fix two real numbers $a<b$. For any $T>0$ define the functional 
\be{eq:F_T}
F_T = \frac{1}{\sqrt{T}}\int_{aT}^{bT} \{g(X_{u+1} - X_u) - \EE[g(Z)]\}\,\dint u\,,
\ee
where $Z\sim\cN(0,1)$ denotes a standard Gaussian random variable. To present the next result, we adopt the $\lesssim$-notation for sequences to the continuous case. In particular, we shall write $a(T)\lesssim b(T)$ for two functions $a,b:(0,\infty)\to\RR$ if the quotient $a(T)/b(T)$ stays bounded for all $T$. For simplicity, we restrict ourself to the case of the Wasserstein distance and do not investigate under which conditions the functionals $F_T$ posess a density.

\begin{proposition}\label{prop:CLT for F_T}
Assume that $\r(0)=1$ and that $\int_{\RR}|\r(u)|\,\dint u<\infty$. Further, suppose that $g:\RR\to\RR$ is a non constant and measurable function, for which $\EE|g(X_1)|^2<\infty$ and $\EE|g((2+\varepsilon)X_1)|^2<\infty$ for some $\varepsilon>0$. Then, as $T\to\infty$, one has that
\be{eq:CLT for F_T}
d_W\left(\frac{F_T}{\sqrt{\Var(F_T)}}, Z\right) \lesssim {T^{-1/2}\over\var(F_T)}\,,
\ee
Moreover, if $g$ is symmetric, then
$$
d_W\left(\frac{F_T}{\sqrt{\Var(F_T)}}, Z\right) \lesssim {T^{-1/2}}\,.
$$
\end{proposition}
\begin{proof}
The proof is almost literally the same as that for Theorem \ref{thm:quantitative CLT}: the sums there have to be replaced by integrals and estimated using the integrability assumption of the covariance function, which corresponds to the case $\alpha>1$ in the proof of Theorem \ref{thm:quantitative CLT}. All combinatorial considerations remain unchanged. We also remark that according to Proposition 6.3 in \cite{NourdinPeccatiReinertSecondOrderPoincare} the symmetry of the function $g$ implies the asymptotic variance $\sigma^2:=\var(F_T)$ exists in $(0,\infty)$, see also Remark \ref{rem:positive variance}. This leads to the second bound. We leave the details to the reader. 
\end{proof}

We would like to emphasize that the central limit theorem for $F_T$ in \cite{NourdinPeccatiReinertSecondOrderPoincare}, which is based on an application of the second-order Poincar\'e inequality on the Wiener space, only works under considerably stronger smoothness assumptions on the function $g$. In particular, $g$ has to be twice continuously differentiable. In this sense, Proposition \ref{prop:CLT for F_T} improves and extends the result of Section 6 in \cite{NourdinPeccatiReinertSecondOrderPoincare}. 

\section{A multivariate extension}\label{sec:Multivariate}

The purpose of this final section is to provide a multivariate extension of Theorem \ref{thm:Main}. To this end, we measure the distance between two $d$-dimensional ($d\geq 2$) random vectors $\bX$ and $\bY$ by the multivariate Wasserstein distance 
$$
d_{mW}(\bX,\bY) := \sup\big|\EE[\varphi(\bX)]-\EE[\varphi(\bY)]\big|\,,
$$
where the supremum is running over all Lipschitz functions $\varphi:\RR^d\to\RR$ with Lipschitz constant less than or equal to $1$. The following result can be seen as the natural multi-dimensional generalization of Theorem \ref{thm:Main}.

\begin{theorem}
Fix $d\geq 2$ and let $C=(C_{ij})_{i,j=1}^d$ be a positive definite $d\times d$ matrix. Suppose that $\bF=(F_1,\ldots,F_d)$ is a centred $d$-dimensional random vector with covariance matrix $C$, and such that $F_i\in L^2(\O)$ and $F_i\in\DD^{1,4}$ for all $i\in\{1,\ldots,d\}$. Further, let $\bZ\sim\cN(0,C)$ be a centred Gaussian random vector with covariance matrix $C$ and denote for each $i\in\{1,\ldots,d\}$ and $q\geq 0$ by $h_{q,i}\in\fH^{\odot q}$ the kernels of the chaotic decomposition of $F_i$. Then
\begin{align*}
d_{mW}(\bF,\bZ) &\leq c\sum_{i,j=1}^d\sum_{p=1}^\infty p\sum_{r=1}^{p-1} (r-1)!{p-1\choose r-1}^2\sqrt{(2(p-r))!}\,\|h_{p,i}\s_r h_{p,j}\|_{\fH^{\otimes 2(p-r)}}\\
&+c\sum_{i,j=1}^d\sum_{p,q=1\atop p\neq q}^\infty p\sum_{r=1}^{\min(p,q)} (r-1)!{p-1\choose r-1}{q-1\choose r-1}\sqrt{(p+q-2r)!}\,\|h_{p,i}\s_r h_{q,j}\|_{\fH^{\otimes p+q-2r}}\,,
\end{align*}
where $c=\sqrt{d}\,\|C^{-1}\|_{\rm op}\|C\|^{1/2}_{\rm op}$ and $\|\,\cdot\,\|_{\rm op}$ indicates the operator norm of the argument matrix.
\end{theorem}
\begin{proof}
As in the proof of Theorem \ref{thm:Main} we assume without loss of generality that $\fH=L^2(A)$ for some $\sigma$-finite non-atomic measure space $(A,\cA,\mu)$, write $\|\,\cdot\,\|_q$ instead of $\|\,\cdot\,\|_{\fH^{\otimes q}}$ and $f_{q,i}$ for $h_{q,i}$, $i\in\{1,\ldots,d\}$. 

From Theorem 6.1.1 in \cite{NourdinPeccatiBook} we have that
$$
d_{mW}(\bF,\bZ) \leq \sqrt{d}\,\|C^{-1}\|_{\rm op}\|C\|^{1/2}_{\rm op}\,\bigg(\sum_{i,j=1}^d\EE\Big[\Big(C_{ij}-\int_A (D_xF_i)(-D_xL^{-1}F_j)\,\mu(\dint x)\Big)^2\Big]\bigg)^{1\over 2}\,.
$$
The covariance representation \eqref{eq:Covariance} as well as the definitions \eqref{eq:DefD} and \eqref{eq:defL-1} of $D$ and $L^{-1}$, respectively, imply that the expectation is bounded by
$$
\EE\Big[\Big(\int_A\sum_{p=1}^\infty pI_{p-1}(f_{i,p}(x,\,\cdot\,))\sum_{q=1}^\infty I_{q-1}(f_{q,j}(x,\,\cdot\,))\,\mu(\dint x)-\sum_{n=1}^\infty n!\langle f_{n,i},f_{n,j}\rangle_n\Big)^2\Big]
$$
and hence
\begin{align*}
d_{mW}(\bF,\bZ) &\leq c\sum_{i,j=1}^d\sum_{p,q=1}^\infty p\Big(\EE\Big[\Big(\int_A I_{p-1}(f_{p,i}(x,\,\cdot\,))I_{q-1}(f_{q,j}(x,\,\cdot\,))\,\mu(\dint x)\\
&\qquad\qquad\qquad\qquad-\EE\int_A I_{p-1}(f_{p,i}(x,\,\cdot\,))I_{q-1}(f_{q,j}(x,\,\cdot\,))\,\mu(\dint x)\Big)^2\Big]\Big)^{1/2}\\
&=c\sum_{i,j=1}^d\sum_{p,q=1}^\infty p\Big(\var\Big(\int_A I_{p-1}(f_{p,i}(x,\,\cdot\,))I_{q-1}(f_{q,j}(x,\,\cdot\,))\,\mu(\dint x)\Big)\Big)^{1/2}
\end{align*}
The variance in the last expression equals $T_1(F,G)-T_2(F,G)^2$ with
\begin{align*}
T_1(F,G) = \int_A\int_A\EE[I_{p-1}(f_{p,i}(x,\,\cdot\,)) & I_{q-1}(f_{q,j}(x,\,\cdot\,))\\
&\times I_{p-1}(f_{p,i}(y,\,\cdot\,))I_{q-1}(f_{q,j}(y,\,\cdot\,))]\,\mu(\dint y)\mu(\dint x)
\end{align*}
and
\begin{align*}
T_2(F,G) = \EE\int_A I_{p-1}(f_{p,i}(x,\,\cdot\,))I_{q-1}(f_{q,j}(x,\,\cdot\,))\,\mu(\dint x)\,.
\end{align*}
Arguing as in the proof of Theorem \ref{thm:Main}, we see that
\begin{align*}
T_1(F,G) = \sum_{r=1}^{\min(p,q)}((r-1)!)^2{p-1\choose r-1}^2{q-1\choose r-1}^2(p+q-2r)!\|f_{p,i}\ts_r f_{q,j}\|_{p+q-2r}^2
\end{align*}
we well as
\begin{align*}
T_2(F,G) = {\bf 1}(p=q)(p-1)!\langle f_{p,i},f_{p,j}\rangle_p\,.
\end{align*}
As a consequence, we find that
\begin{align*}
& T_1(F,G)-T_2(F,G)^2 \\
&= {\bf 1}(p=q)\sum_{r=1}^{p-1}((r-1)!)^2{p-1\choose r-1}^4(2(p-r))!\|f_{p,i}\ts_r f_{p,j}\|_{2(p-r)}^2\\
&\qquad +{\bf 1}(p\neq q)\sum_{r=1}^{\min(p,q)}((r-1)!)^2{p-1\choose r-1}^2{q-1\choose r-1}^2(p+q-2r)!\|f_{p,i}\ts_r f_{q,j}\|_{p+q-2r}^2
\end{align*}
and finally
\begin{align*}
d_{mW}(\bF,\bZ) & \leq c\sum_{i,j=1}^d\sum_{p=1}^\infty p\sum_{r=1}^{p-1}(r-1)!{p-1\choose r-1}^2\sqrt{(2(p-r))!} \|f_{p,i}\s_r f_{p,j}\|_{2(p-r)}\\
& + c\sum_{i,j=1}^d\sum_{p,q=1\atop p\neq q}^\infty p\sum_{r=1}^{\min(p,q)}(r-1)!{p-1\choose r-1}{q-1\choose r-1}\sqrt{(p+q-2r)!}\|f_{p,i}\s_r f_{q,j}\|_{p+q-2r}\,.
\end{align*}
This completes the proof.
\end{proof}

\section{Appendix}
Let $\bar{\mathcal B}(\RR^n)$ be the completion of the Borel $\sigma$-field $\cB(\RR^n)$ with respect to the $n$-dimensional Lebesgue measure $\l^n\colon \bar \cB(\RR^n)\to[0,\infty]$. We call a $\bar \cB(\RR)$-measurable function $f\colon \RR^n\to\RR^n$ $0$-measure-preserving if $\l^n(B)=0$ for all $B\in\bar\cB(\RR^n)$ implies that $\l^n(f^{-1}(B))=0$.

\begin{lemma}
Let $g\colon \RR\to\RR$ be a $\bar \cB(\RR)$-measurable function. For fixed $n\ge1$ define the measurable function $g_n\colon \RR^n\to\RR^n$ by applying $g$ to each coordinate, that is, $g_n(x_1, \ldots, x_n) := (g(x_1), \ldots, g(x_n))$, $(x_1,\ldots,x_n)\in\RR^n$. Then $g$ is $0$-measure-preserving if and only if $g_n$ is $0$-measure-preserving.
\end{lemma}

\begin{proof} If $n=1$, then there is nothing to show. So, we let $n>1$ and assume that $g$ is $0$-measure-preserving. Now, the assertion follows by induction on $n$. For this reason, it is sufficient to restrict to the case that $n=2$. For any $B\in\bar\cB(\RR^2)$ and any $x\in\RR$ we write $B_x:=\{y\in\RR\colon (x,y)\in B\}$, which is an element of $\bar\cB(\RR)$. 

Choose a $B\in\bar \cB(\RR^2)$ with $\l^2(B)=0$. Then 
\[
\l^2(B) = \int_{\RR} \l^1(B_x)\,\l^1(\dint x)=0\,.
\]
As a consequence, the set $N:=\{x\in\RR\colon \l^1(B_x)>0\}$ satisfies $\l^1(N)=0$. Now, let $A:=g_2^{-1}(B)$. Then, for each $y\in\RR$, one has that
\begin{align*}
A_y&=\{z\in\RR\colon (y,z)\in A\} 
=\{z\in\RR\colon (g(y), g(z))\in B\} \\
&=\{z\in\RR\colon g(z) \in B_{g(y)}\} = g^{-1}(B_{g(y)})\,.
\end{align*}
Due to our assumptions on $g$ it holds that $\{y\in\RR\colon \l^1(A_y)>0\} = \{y\in\RR\colon \l^1(B_{g(y)})>0\} = g^{-1}(N)$, which is a set of Lebesgue measure $0$. Now, the claim follows upon observing that 
\[
\l^2(A) = \int_{\RR}\l^1(A_y)\,\l^1(\dint y) =\int_{g^{-1}(N)}\l^1(A_y)\,\l^1(\dint y)=0\,.
\]
On the other hand, if $g_n$ is $0$-measure-preserving, then $g$ is also $0$-measure-preserving by considering sets of the form $B_1 \times \cdots \times B_n$ for $B_i\in\bar\cB(\RR)$ where $\l^1(B_1) = \cdots = \l^1(B_n)=0$.
\end{proof}

\begin{lemma}\label{eq:abs cont0}
Let $\mu$ be a measure on $\bar\cB(\RR^n)$, which is equivalent to the Lebesgue measure $\l^n$ and let $f\colon \RR^n\to\RR^n$ be measurable function. Then $\mu\circ f^{-1}$ is absolutely continuous with respect to $\l^n$ if and only if $f$ is $0$-measure-preserving. 
\end{lemma}
\begin{proof}
The proof is standard.
\end{proof}

\begin{lemma}
If a random vector $(X_1, \ldots, X_n)$ has a density with respect to $\l^n$ for some $n\in\NN$, then also $X_1 + \cdots + X_n$ has a density with respect to $\l^1$.
\end{lemma}
\begin{proof}
The claim follows by means of the transformation theorem.
\end{proof}

\section*{Acknowledgement}
We would like to thank Johanna F.~Ziegel and Anja M\"uhlemann for inspiring discussions. TF acknowledges funding by the Swiss National Science Foundation (SNF) via grant 152609. 

\bibliographystyle{siam}
\bibliography{Biblio_InfiniteChaos2}

\begin{thebibliography}{10}

\bibitem{AzaisRandTrigPol}
{\sc J.-M. Aza\"is, F.~Dalmao, and J.~R. Le\'on}, {\em {CLT} for the zeros of
  classical random trigonometric polynomials}, Ann. Inst. H. Poincar\'e Probab.
  Statist., 52 (2016), pp.~804--820.

\bibitem{BarndorffNielsenEtAl1}
{\sc O.~E. Barndorff-Nielsen, J.~M. Corcuera, and M.~Podolskij}, {\em Power
  variation for {G}aussian processes with stationary increments}, Stochastic
  Process. Appl., 119 (2009), pp.~1845--1865.

\bibitem{BiermeBonamiLeon2011}
{\sc H.~Bierm{\'e}, A.~Bonami, and J.~R. Le{\'o}n}, {\em Central limit theorems
  and quadratic variations in terms of spectral density}, Electron. J. Probab.,
  16 (2011), pp.~362--395.

\bibitem{EstradeLeon}
{\sc A.~Estrade and J.~Leon}, {\em A central limit theorem for the {E}uler
  characteristic of a {G}aussian excursion set}, to appear in Ann. Probab.,
  (2016+).

\bibitem{HugLastSchulte}
{\sc D.~Hug, G.~Last, and M.~Schulte}, {\em Second-order properties and central
  limit theorems for geometric functionals of {B}oolean models}, Ann. Appl.
  Probab., 26 (2016), pp.~73--135.

\bibitem{Janson}
{\sc S.~Janson}, {\em Gaussian {H}ilbert {S}paces}, vol.~129 of Cambridge
  Tracts in Mathematics, Cambridge University Press, Cambridge, 1997.

\bibitem{NourdinPeccati09}
{\sc I.~Nourdin and G.~Peccati}, {\em Stein's method on {W}iener chaos},
  Probab. Theory Related Fields, 145 (2009), pp.~75--118.

\bibitem{NourdinPeccatiBook}
{\sc I.~Nourdin and G.~Peccati}, {\em {N}ormal {A}pproximations with
  {M}alliavin {C}alculus - {F}rom {S}tein's {M}ethod to {U}niversality},
  Cambridge University Press, Cambridge, new.~ed., 2012.

\bibitem{NourdinPeccatiOptimalRates}
{\sc I.~Nourdin and G.~Peccati}, {\em The optimal fourth moment theorem}, Proc.
  Amer. Math. Soc., 143 (2015), pp.~3123--3133.

\bibitem{NourdinPeccatiPodolskij}
{\sc I.~Nourdin, G.~Peccati, and M.~Podolskij}, {\em Quantitative
  {B}reuer-{M}ajor theorems}, Stochastic Process. Appl., 121 (2011),
  pp.~793--812.

\bibitem{NourdinPeccatiReinertSecondOrderPoincare}
{\sc I.~Nourdin, G.~Peccati, and G.~Reinert}, {\em Second order {P}oincar\'e
  inequalities and {CLT}s on {W}iener space}, J. Funct. Anal., 257 (2009),
  pp.~593--609.

\bibitem{Nualart}
{\sc D.~Nualart}, {\em The {M}alliavin {C}alculus and {R}elated {T}opics},
  Probability and its Applications (New York), Springer-Verlag, Berlin,
  second~ed., 2006.

\bibitem{PeccatiTaqquBook}
{\sc G.~Peccati and M.~S. Taqqu}, {\em Wiener {C}haos: {M}oments, {C}umulants
  and {D}iagrams}, vol.~1 of Bocconi \& Springer Series, Springer, Milan;
  Bocconi University Press, Milan, 2011.

\end{thebibliography}

\end{document}